\documentclass[11pt]{amsart} 
\usepackage{amsfonts} 
\usepackage{graphicx}
\usepackage{color}
\usepackage{hyperref}
\hypersetup{backref,colorlinks=true,citecolor=blue,linkcolor=blue}

\newtheorem{theorem}{Theorem}[section]
\newtheorem{proposition}[theorem]{Proposition}

\theoremstyle{definition}
\newtheorem{definition}[theorem]{Definition}

\renewenvironment{proof}{{\noindent\bf Proof.}}{\hfill $\Box$\par\vskip3mm}

\begin{document}

\title{Embedded annuli and Jones' conjecture}

\author{Douglas J. LaFountain}
\address{Department of Mathematics \\ Western Illinois University}
\email{d-lafountain@wiu.edu}

\author{William W. Menasco}
\address{Department of Mathematics \\ University at Buffalo}
\email{menasco@buffalo.edu}

\begin{abstract}
We show that after stabilizations of opposite parity and braid isotopy, any two braids in the same topological link type cobound embedded annuli.  We use this to prove the generalized Jones conjecture relating the braid index and algebraic length of closed braids within a link type, following a reformulation of the problem by Kawamuro. 
\end{abstract}

\maketitle

\section{Introduction}

\label{sec:intro}

Consider an oriented unknotted braid axis $A$ in $S^3$ whose complement fibers over $S^1$ with oriented disc fibers $\{A(\theta)\}_{\theta\in S^1}$.  For a given oriented topological link type $\mathcal{L}$ with $m$ components we study closed braid representatives $\beta \in \mathcal{L}$ which are embedded in $S^3 \setminus A$ and which transversely intersect the $A(\theta)$-disc fibers positively.  Our first result is the following: 

\begin{proposition}
\label{prop:annuli}
Let $\beta_1, \beta_2 \in \mathcal{L}$ be braided about a common unknotted braid axis in $S^3$.  Then there exists two braids $\hat \beta_1$ and $\hat \beta_2$ whose $m$ components pairwise cobound $m$ embedded annuli, such that $\hat \beta_1$ is obtained from $\beta_1$ via negative stabilizations, and $\hat \beta_2$ is obtained from $\beta_2$ via positive stabilizations and braid isotopy.
\end{proposition}

We remark that this final braid isotopy will in general involve $\hat\beta_2$ passing through $\hat\beta_1$ so as to obtain the embedded annuli.

For any closed braid $\beta$, we denote its braid index by $n(\beta)$ and its algebraic length by $\ell(\beta)$.  We use the above proposition to prove the following theorem, namely the generalized Jones conjecture which relates the braid index and algebraic length of closed braids within a topological link type \cite{[J],[K1]}.

\begin{theorem}[Jones' conjecture]
\label{thm:jones}
Let $\beta, \beta_0 \in \mathcal{L}$ be two closed braids such that $n(\beta_0)$ is minimal for its link type.  Then $$|\ell(\beta)-\ell(\beta_0)| \leq n(\beta) - n(\beta_0)$$

\end{theorem}

As noted by others, the veracity of Jones' conjecture yields immediate applications to the study of transverse links in the contact 3-sphere, quasi-positive and strongly quasi-positive braids, representations of braid groups and polynomial invariants for links \cite{[DP],[E],[K1],[K],[S]}.  Recently, Dynnikov and Prasolov provided a proof of Jones' conjecture by studying bypasses for rectangular diagrams representing Legendrian links \cite{[DP]}.  We present here an independent proof using an alternative approach.

The outline of the paper is as follows:  in Section \ref{sec:annuli} we briefly review background, and then prove Proposition \ref{prop:annuli} using braid foliation techniques of Birman-Menasco; in particular, our argument is inspired by the constructions in \cite{[BM3]} and \S2 of \cite{[BM2]}.  In Section \ref{sec:jones} we then establish Theorem \ref{thm:jones} by proving an equivalent statement proposed by Kawamuro \cite{[K1],[K]}. %and in Section \ref{sec:apps} we list as corollaries various applications of Jones' conjecture.

\medskip

\noindent\textbf{Acknowledgements.}  The authors would like to thank Joan Birman and Keiko Kawamuro for their helpful comments while reading a preliminary version of the current paper.

\section{Stabilizing to embedded annuli}
\label{sec:annuli} 

Let $\beta$ be braided about an unknotted braid axis $A$ with braid fibration $\{A(\theta)\}_{\theta \in S^1}$.  The {\em braid index} of $\beta$, denoted $n(\beta)$, is the number of intersections of $\beta$ with any $A(\theta)$-disc. The {\em algebraic length} of $\beta$, denoted $\ell(\beta)$, is the sum of the signed crossings in any regular braid projection of $\beta$.  Given $\beta$, we may alter it through standard moves, namely: {\em braid isotopy} in the complement of $A$ which does not change $n$ nor $\ell$; {\em exchange moves} which change neither $n$ nor $\ell$ (see the left side of Figure \ref{fig:basicbraidmovesB}); and {\em stabilization (destabilization)} which increases (decreases) $n$ and either increases or decreases $\ell$ depending on whether the stabilization or destabilization is positive or negative (see the right side of Figure \ref{fig:basicbraidmovesB}).

\begin{figure}[htbp]
	\centering
		\includegraphics[width=0.70\textwidth]{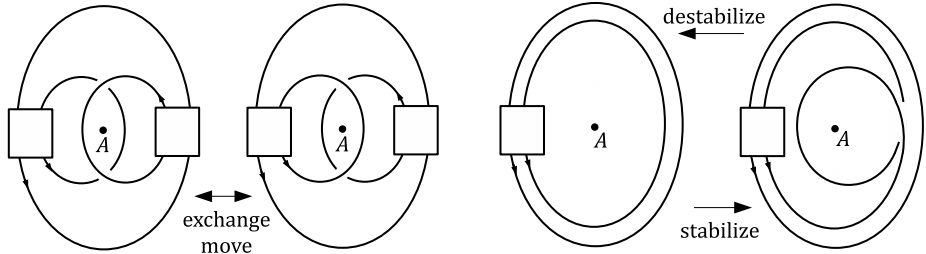}
	\caption{{\small On the left is an exchange move; on the right is stabilization (destabilization).}}
	\label{fig:basicbraidmovesB}
\end{figure}

Throughout this paper we will be studying braids using various embedded oriented annuli and bigon discs.  By general position we may assume that $A$ intersects our surface of interest $S$ transversely in finitely many points, called {\em vertices}, of either positive or negative parity, depending on whether the orientation of $A$ agrees or disagrees with the orientation of $S$ at that vertex.  Provided the boundary components, or boundary arcs, of our surface $S$ are transverse to $A(\theta)$-discs in the braid fibration, by standard braid foliation arguments \cite{[BF]} we may also assume that there are finitely many points of tangency between $S$ and the $A(\theta)$-disc fibers, all of which yield simple saddle singularities in the foliation of $S$ induced by the braid fibration.  We will generically refer to such points as {\em singularities}, and again these will be of positive or negative parity depending on whether the orientation of $S$ agrees or disagrees with the orientation of $A(\theta)$ at those points.

Also by standard arguments \cite{[BF]}, the non-singular leaves in the $A(\theta)$-foliation of our discs and annuli will either be {\em s-arcs} whose endpoints are on two different boundary components of an annulus (or different boundary arcs of a bigon disc); {\em a-arcs} with one endpoint on a vertex and one endpoint on a boundary component or arc; or {\em b-arcs} whose endpoints are on two different vertices of opposite parity.  Furthermore, singular leaves can then be classified as to what non-singular leaves interact to form the saddle singularity; specifically, in general we will have $aa$-, $bb$-, $ab$-, $as$- and $abs$-singularities.  These are depicted in Figure \ref{fig:tilesB}, where the oriented transverse boundary is given by bold black arrows; we will refer to the gray shaded regions as $aa$-, $bb$-, $ab$-, $as$- and $abs$-tiles, classified by the singularity which they contain.   

\begin{figure}[htbp]
	\centering
		\includegraphics[width=0.60\textwidth]{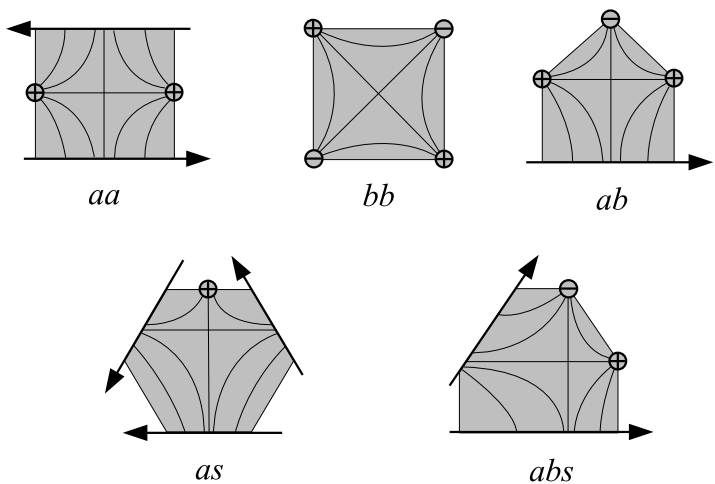}
	\caption{{\small The five types of saddle singularities, classified by the types of non-singular arcs interacting to form the singularity.  The parities of the vertices may be reversed provided the orientations of the transverse boundary arcs are compatibly reversed.}}
	\label{fig:tilesB}
\end{figure}

The {\em valence} of a vertex is the number of singular leaves for which it serves as an endpoint; alternatively the valence is the number of (one-parameter families of) $a$-arcs or $b$-arcs to which the vertex is adjacent.  Any vertex may therefore be labeled as type $(\alpha,\beta)$ where $\alpha$ and $\beta$ are the number of $a$-arcs and $b$-arcs to which the vertex is adjacent, with the valence of the vertex being $\alpha + \beta$. 

As discussed in detail in \cite{[BF]} and shown in Figure \ref{fig:valence123B}, valence-1 vertices of type $(1,0)$ indicate a destabilization of the braided boundary of a surface which results in the elimination of that vertex; valence-2 vertices of type $(1,1)$ or $(0,2)$ indicate an exchange move with a corresponding elimination of two vertices; and valence-3 vertices of type $(0,3)$ indicate the presence of two consecutive singularities of like parity adjacent to that vertex whose $A(\theta)$-order can be interchanged using braid isotopy, thus reducing the valence of that vertex to 2. This re-ordering of consecutive singularities of like parity is referred to as a {\em standard change of fibration} -- see the bottom right picture in Figure \ref{fig:valence123B}.  As a result, using destabilizations, exchange moves and braid isotopy we can reduce the total number of vertices provided there are such vertices of valence-1, -2 or -3.

\begin{figure}[htbp]
	\centering
		\includegraphics[width=0.70\textwidth]{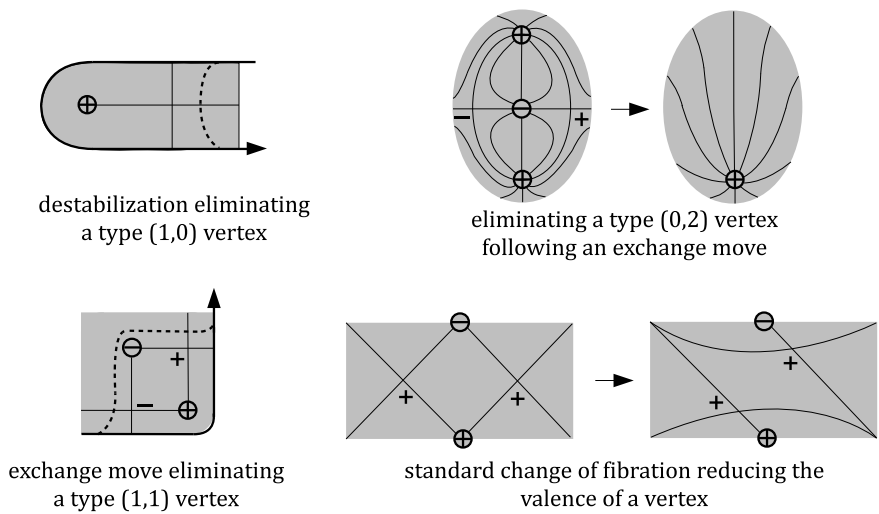}
	\caption{{\small Eliminating, and reducing the valence of, vertices using destabilization, exchange moves and braid isotopy.  In the figures on the left, the braided boundary is indicated by the bold black arrow, and the subdisc cobounded by it and the dashed line is eliminated following destabilization and an exchange move, respectively.}}
	\label{fig:valence123B}
\end{figure}

With this background, we can now proceed with the proof of Proposition \ref{prop:annuli}, namely that given braids $\beta_1, \beta_2 \in \mathcal{L}$, then after negative stabilizations of $\beta_1$ and positive stabilizations of $\beta_2$, we obtain two braids $\hat\beta_1, \hat\beta_2$ which cobound $m$ embedded annuli, one for each component of the link type $\mathcal{L}$.  

\bigskip

\begin{proof}
We consider $\beta_1$ and $\beta_2$ braided about a common axis $A$, and furthermore think of $A$ as the $z$-axis in $\mathbb{R}^3$, with $\beta_1 \subset \mathbb{R}^3_- = \{(x,y,z) \in \mathbb{R}^3 | z < 0\}$ and $\beta_2 \subset \mathbb{R}^3_+ = \{(x,y,z) \in \mathbb{R}^3 | z > 0\}$.  We begin with $\beta_1$ and consider a braided push-off $\beta'_1$ of $\beta_1$ such that $\beta'_1,\beta_1 \subset \mathbb{R}^3_-$ and $\beta'_1 \sqcup \beta_1$ cobound $m$ embedded annuli, one for each component of $\mathcal{L}$.

Now consider a regular braid projection of $\beta'_1 \sqcup \beta_1$ onto the $z=-1$-plane in $\mathbb{R}^3_-$, and wherever there is a double point at which $\beta'_1$ passes under $\beta_1$, we imagine isotoping $\beta'_1$ locally upwards in the $z$-direction through $\beta_1$.  The result is $\beta'_2$ which is in fact braid isotopic to $\beta'_1$, but which is now unlinked from $\beta_1$, yet still in $\mathbb{R}^3_-$.  Furthermore, for each of these $r$ crossing changes, we have a bigon disc $D_i$, $1 \leq i \leq r$, cobounded by arcs of $\beta'_1$ and $\beta'_2$ and intersected once by $\beta_1$; see Figure \ref{fig:passthroughB}.   %We can choose a framing for each of these $\alpha_i$, namely a choice of section in the normal bundle for the $\alpha_i$; furthermore, we observe that $\alpha_i$ does not intersect the interiors of the embedded annuli.

\begin{figure}[htbp]
	\centering
		\includegraphics[width=0.75\textwidth]{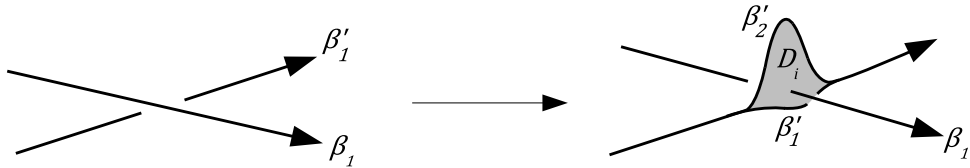}
	\caption{{\small An under-crossing of $\beta'_1$ passes through $\beta_1$ to produce $\beta'_2$, and produces a bigon disc $D_i$ which is intersected once by $\beta_1$.} }
	\label{fig:passthroughB}
\end{figure}

We may now vertically braid isotop $\beta'_2$ in the positive $z$-direction so that $\beta'_2 \subset \mathbb{R}^3_+$.  Then, since $\beta'_2, \beta_2 \in \mathcal{L}$, there is an ambient isotopy of $\mathbb{R}^3_+$, relative to the $xy$-plane, which takes $\beta'_2$ to $\beta_2$.  The braid $\beta'_1$ will experience an induced isotopy to a link $L'_1$ which continues to cobound with $\beta_1$ a total of $m$ embedded annuli.  Similarly, the discs $D_i$ will experience an induced isotopy.  The boundary of each new $D_i$ consists of two arcs, namely $\partial_i^+ \subset \beta_2$ and the boundary arc $\partial_i^- \subset L'_1$; see Figure \ref{fig:isotopyB}.  Observe that the braid $\beta_1$ intersects each disc $D_i$ once close to the $\partial_i^-$ boundary arc in $\mathbb{R}^3_-$.  Since $\partial_i^+ \subset \beta_2$ it is transverse to the $\{A(\theta)\}$-fibration.  However, the $\partial_i^-$ boundary arc will not necessarily be transverse to the $\{A(\theta)\}$-fibration in $\mathbb{R}^3_+$, although it will still be transverse to the portions of $A(\theta)$-discs in $\mathbb{R}^3_-$.

\begin{figure}[htbp]
	\centering
		\includegraphics[width=0.60\textwidth]{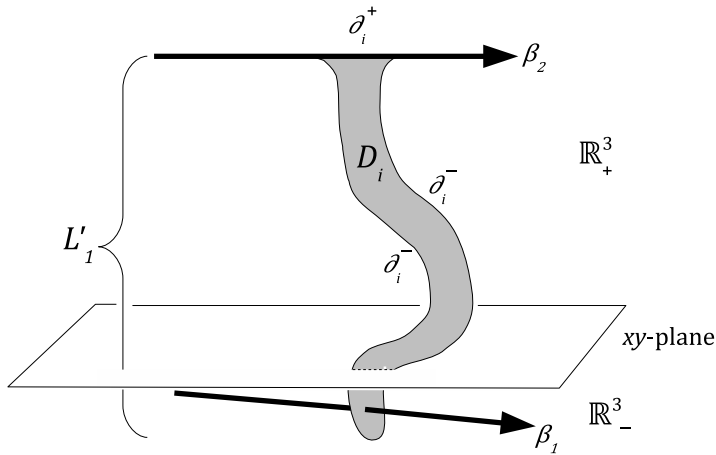}
	\caption{{\small The link $L'_1$ and discs $D_i$ after isotoping $\beta'_2$ to $\beta_2$ in $\mathbb{R}^3_+$.}}
	\label{fig:isotopyB}
\end{figure}

We now apply Alexander's theorem for links in $\mathbb{R}^3_+$ so as to make the $\partial_i^-$ boundary arcs for the discs $D_i$ transverse to the $\{A(\theta)\}$-fibration, and we can do so without changing the conjugacy class of $\beta_2$ (\cite{[A]}, see also \S 2 of \cite{[BM2]}).  This takes the link $L'_1$ to a braid $\beta''_1$ which continues to cobound with $\beta_1$ a total of $m$ embedded annuli.  Moreover, we can now realize bigon discs $D_i$ whose braid foliations may be assumed to consist of bands of $s$-arcs alternating with regions tiled by $aa$-, $ab$- and $bb$-tiles, with either a single $as$-tile serving as a transition between the $s$-band and the tiled region, or two $abs$-tiles serving as the transition on either end of a bigon disc.  We orient each $D_i$ so as to agree with the orientation of the $\partial_i^-$ boundary arc, so that $a$-arcs along $\partial_i^-$ connect to positive vertices, and $a$-arcs along $\partial_i^+$ connect to negative vertices; we then work to simplify the foliation of each $D_i$.  

First, consider all singular leaves which intersect the $\partial_i^-$ boundary arc; by slightly perturbing $D_i$ if necessary we may assume $\beta_1$ does not intersect any such singular leaves.  If any of these singular leaves has its other endpoint on a vertex $v$, we may stabilize the $\partial_i^-$ arc along that singular leaf so as to remove that vertex $v$ from the foliation of $D_i$ (see Figure \ref{fig:stabilizesingB}).

\begin{figure}[htbp]
	\centering
		\includegraphics[width=0.65\textwidth]{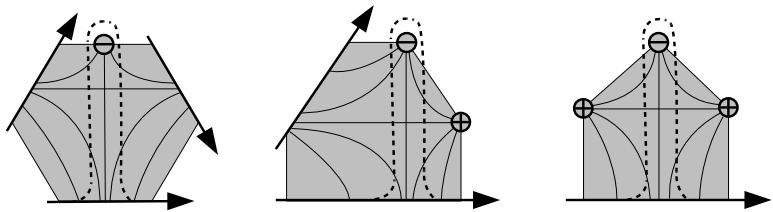}
	\caption{{\small We can stabilize $\partial_i^-$ along singular leaves if any of these three configurations occur.}}
	\label{fig:stabilizesingB}
\end{figure}

Thus we may assume that all singular leaves which intersect the $\partial_i^-$ boundary arc in fact intersect it twice.  As a result, if we consider the graph in $D_i$ which consists of singular leaves that connect vertices to vertices, or vertices to one of the boundary arcs  $\partial_i^-$ or $\partial_i^+$, this graph must be a forest, with the root of each tree on the $\partial_i^+$ boundary arc.  In other words, the foliation of each $D_i$ consists of a band of $s$-arcs with trees of $aa$-tiles extending off of it along the $\partial_i^-$ boundary arc, with a single $as$-tile connecting each tree of the forest to the band of $s$-arcs -- see the left-hand side of Figure \ref{fig:treeB}.

\begin{figure}[htbp]
	\centering
		\includegraphics[width=0.65\textwidth]{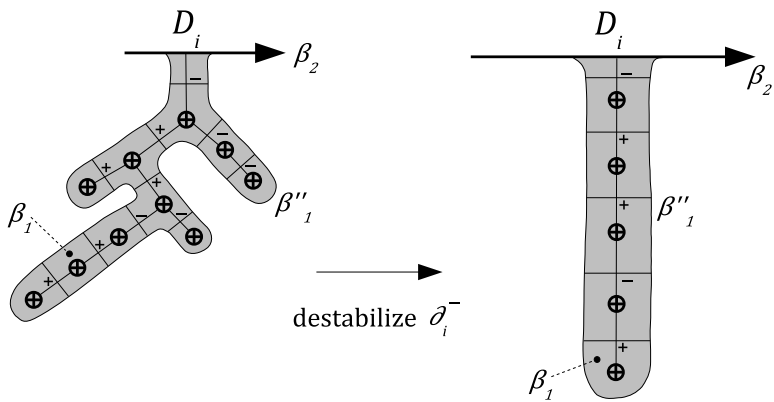}
	\caption{{\small On the left is $D_i$ whose vertex-singularity graph is a tree after stabilizing $\partial_i^-$; on the right is $D_i$ whose graph is a linear tree after destabilizing $\partial_i^-$.}}
	\label{fig:treeB}
\end{figure}

We can now destabilize the $\partial_i^-$ boundary arc along each outermost $aa$-tile containing a valence-1 vertex, as long as it does not contain the point of intersection with $\beta_1$; the result is that each $D_i$ may be assumed to be a linear string of $aa$-tiles as in the right-hand side of Figure \ref{fig:treeB}, where the outermost tile farthest from $\beta_2$ contains the intersection with $\beta_1$.  Observe that throughout this simplification of the foliation, both $\beta_2$ and $\beta_1$ are fixed, and the resulting braid $\beta''_1$ continues to cobound with $\beta_1$ a total of $m$ embedded annuli.

We now examine the resulting linear foliation on a single $D_i$; it consists of positive vertices, along with a sequence of singularities.  The result is a twisted band, with the parity of the half-twists given by the signs of the singularities.  It is then evident that if anytime a negative singularity is consecutive with a positive singularity, we may perform an exchange move so as to re-order those two singularities -- see Figure \ref{fig:exchangesB}.  Moreover, this exchange move involves an isotopy of $D_i$ which is performed in a regular neighborhood of the sub-disc of $D_i$ cobounded by $\partial_i^-$ and the singular leaves associated with the two singularities; thus the isotopy fixes both $\beta_1$, $\beta_2$ and the other $D_i$.

\begin{figure}[htbp]
	\centering
		\includegraphics[width=0.50\textwidth]{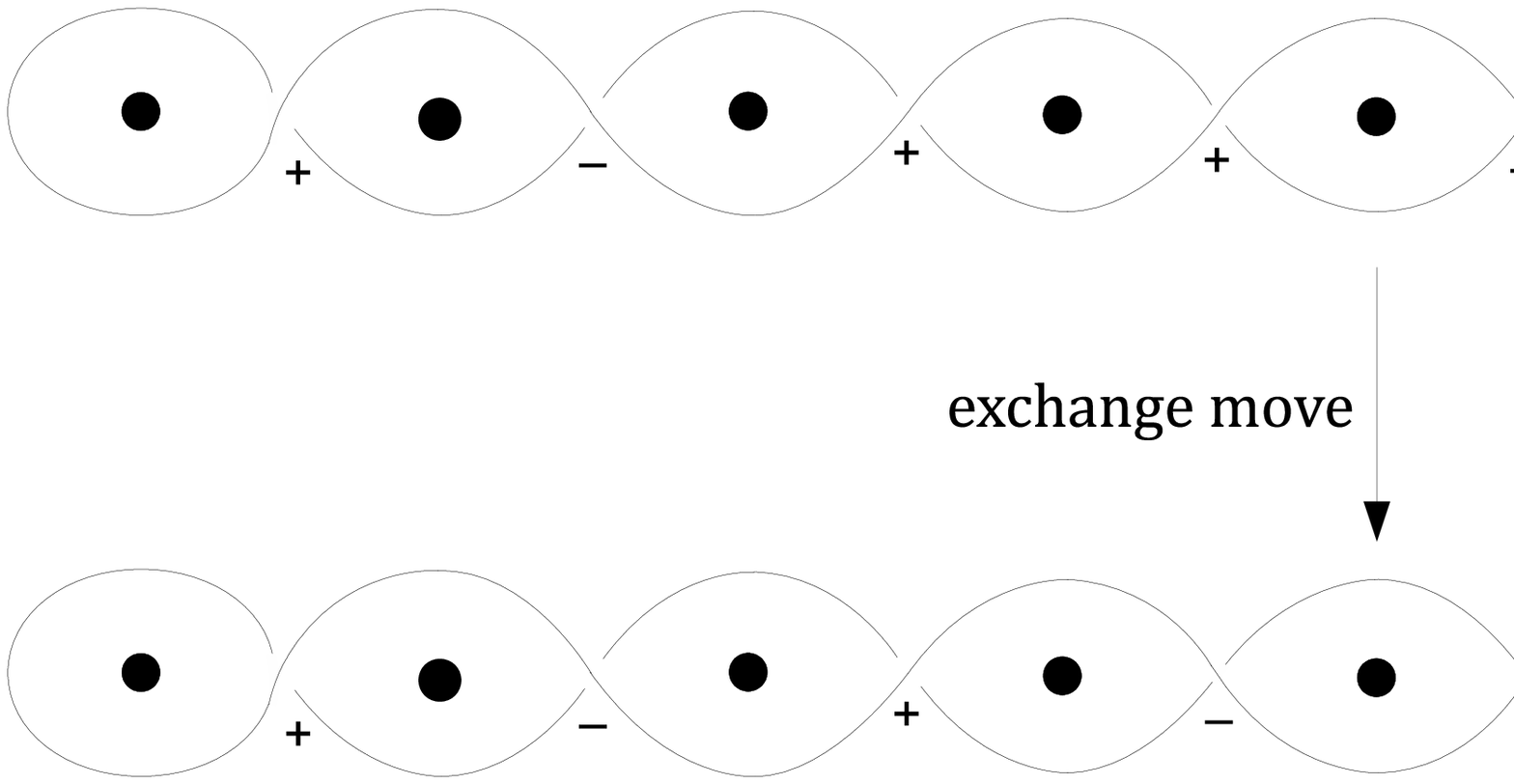}
	\caption{{\small An exchange move involving the two consecutive singularities closest to the $\beta_2$-end of the $D_i$ linear tree.  The black dots represent intersections of the braid axis $A$ with $D_i$.}}
	\label{fig:exchangesB}
\end{figure}

In this way, we may arrange that all positive singularities are stacked at the $\beta_2$-end of the foliation of $D_i$, and all negative singularities are stacked at the $\beta_1$-end of the foliation of $D_i$.  We then can negatively destabilize $\partial_i^-$ through those singularities, therefore inducing negative stabilizations of $\beta_1$ to arrive at $\hat\beta_1$, as depicted in Figures  \ref{fig:destabstabB} and \ref{fig:microflypeB}.  (See `microflypes' in \S2.3 of \cite{[BM2]}.)  We then observe that stabilizing $\beta_2$ along the remaining positive singularities and allowing it to pass through $\hat\beta_1$ will yield $\hat\beta_2$ which cobounds with $\hat\beta_1$ a total of $m$ embedded annuli.
\end{proof}

\begin{figure}[htbp]
	\centering
		\includegraphics[width=0.75\textwidth]{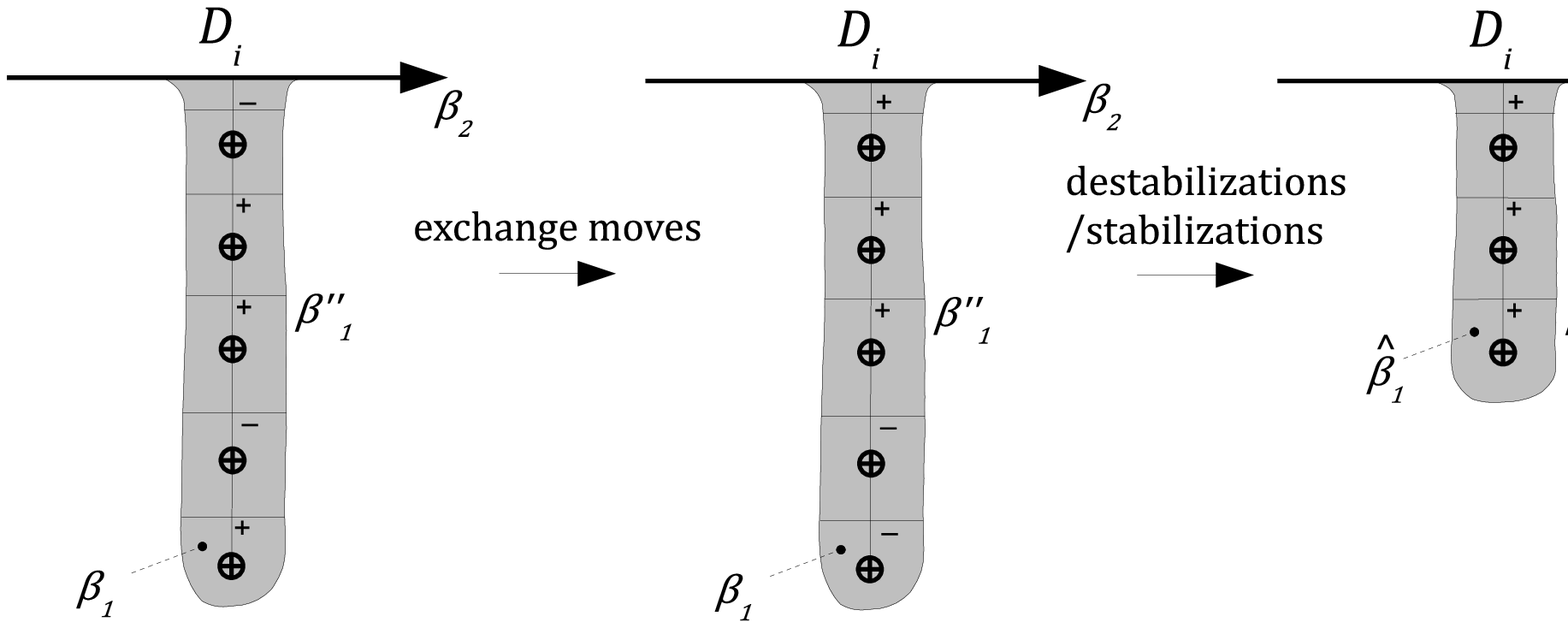}
	\caption{{\small Shown are the negative destabilizations of $\partial_i^-$ along with corresponding negative stabilizations of $\beta_1$ which yield $\hat\beta_1$.}}
	\label{fig:destabstabB}
\end{figure}

\begin{figure}[htbp]
	\centering
		\includegraphics[width=0.70\textwidth]{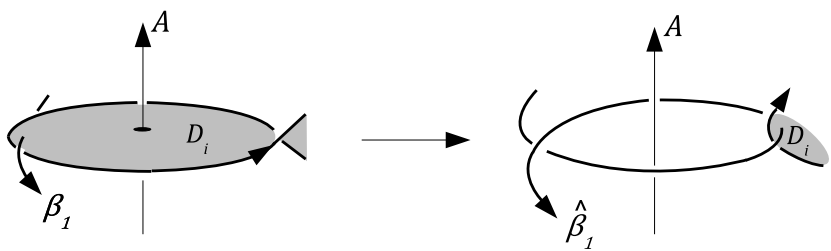}
	\caption{{\small A destabilization of the braided boundary of $D_i$ induces a stabilization of $\beta_1$.}}
	\label{fig:microflypeB}
\end{figure}

\section{Jones' conjecture}
\label{sec:jones}

Before proving Theorem \ref{thm:jones}, we recall a reformulation of the generalized Jones conjecture as described by Kawamuro \cite{[K1],[K]}.  To do so, observe that for any braid $\beta \in \mathcal{L}$ there is an ordered pair $(\ell(\beta),n(\beta))$; we then have the following definition:

\begin{definition}
\label{defn:cone}
Let $\beta \in \mathcal{L}$.  The {\em cone} of $\beta$ is defined to be $(\ell(\beta),n(\beta))$ along with all pairs $(\ell,n)$ that can be achieved by stabilizing $\beta$.
\end{definition}

In the $(\ell,n)$-plane, the cone of $\beta$ indeed has the shape of a cone; it consists of a lattice of points at or above $(\ell(\beta),n(\beta))$ which are contained within the region bounded by lines of slope +1 and -1 passing through the point $(\ell(\beta),n(\beta))$; see Figure \ref{fig:coneB}.

\begin{figure}[htbp]
	\centering
		\includegraphics[width=0.55\textwidth]{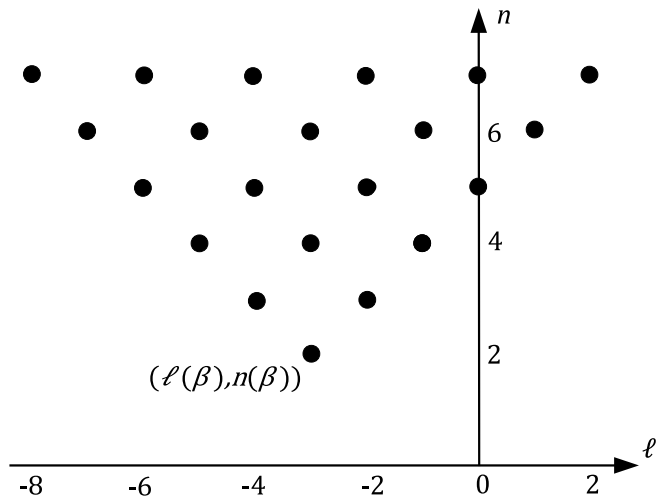}
	\caption{{\small The cone of a braid $\beta$ in the $(\ell,n)$-plane.}}
	\label{fig:coneB}
\end{figure}

An equivalent reformulation of Theorem \ref{thm:jones} due to Kawamuro is then the following proposition, which we prove.

\begin{proposition}
Let $\beta_0 \in \mathcal{L}$ be at minimum braid index for the link type $\mathcal{L}$.  If $\beta \in \mathcal{L}$, then the cone of $\beta$ is contained in the cone of $\beta_0$.
\end{proposition}

\begin{proof}
Suppose for contradiction that there is a $\beta_1 \in \mathcal{L}$ whose cone contains points outside the cone of $\beta_0$.  Then in fact $(\ell(\beta_1),n(\beta_1))$ must be outside the cone of $\beta_0$.  We assume for the moment that $(\ell(\beta_1),n(\beta_1))$ is such that $\ell(\beta_1) < \ell(\beta)$ for any $\beta$ in the cone of $\beta_0$ with $n(\beta_1)=n(\beta)$.  In other words, as we look at the cone of $\beta_0$, we see $\beta_1$ {\em to the left} of the cone of $\beta_0$ -- see Figure \ref{fig:twoconesB}.

By Proposition \ref{prop:annuli} there exists $\hat\beta_0, \hat\beta_1 \in \mathcal{L}$ such that the $m$ components of each braid pairwise cobound $m$ embedded annuli, where $\hat\beta_0$ is obtained from $\beta_0$ by positively stabilizing, and $\hat\beta_1$ is obtained from $\beta_1$ by negatively stabilizing.  As a result, we observe that $\hat\beta_1$ lies {\em to the left and outside of} the cone of $\beta_0$, and $\hat\beta_0$ lies {\em to the right and outside of} the cone for $\beta_1$; see Figure \ref{fig:twoconesB}.  Our goal in the proof of the current proposition is to use the embedded annuli to find a braid $\hat\beta_0^*$ which is obtained from $\hat\beta_0$ by destabilizations, braid isotopy and exchange moves, and a braid $\hat\beta_1^*$ which is obtained from $\hat\beta_1$ by destabilizations, braid isotopy and exchange moves, such that the cone of $\hat\beta_0^*$ equals the cone of $\hat\beta_1^*$.  Since neither braid isotopy nor exchange moves changes the algebraic length or braid index of a braid, the conclusion is that $n(\hat\beta_0^*)=n(\hat\beta_1^*) < n(\beta_0)$, which is our desired contradiction since $\beta_0$ is at minimum braid index.  It is then evident that if we begin with $\beta_1$ {\em to the right} of the cone of $\beta_0$, we can reverse the roles of $\beta_0, \beta_1$ in  Proposition \ref{prop:annuli} to achieve a similar contradiction.

\begin{figure}[htbp]
	\centering
		\includegraphics[width=0.65\textwidth]{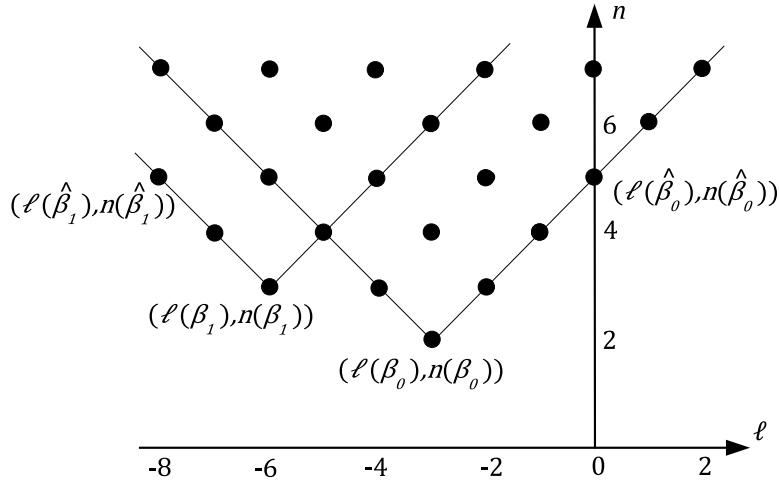}
	\caption{{\small Two cones with neither one contained in the other.}}
	\label{fig:twoconesB}
\end{figure}

So it remains to consider a representative annulus $\mathcal{A}$ cobounded by a component of $\hat\beta_0$ and a component of $\hat\beta_1$.  First, we observe that any $aa$- or $as$- tile must separate from $\mathcal{A}$ a subdisc $\Delta$ which is cobounded by the associated singular leaf and a subarc of a single component of $\partial\mathcal{A}$.  By an Euler characteristic calculation (Lemma 7 of \cite{[BM]}), we are guaranteed that $$V(1,1) + 2V(0,2) + V(0,3) \geq 4$$ where $V(\alpha,\beta)$ is the number of vertices of type $(\alpha,\beta)$ in $\Delta$.  Therefore, the tiling of $\Delta$ either contains valence-1 vertices which we can remove by destabilizing, or valence-2 or valence-3 vertices which can then be eliminated after braid isotopy and exchange moves.  We may therefore assume that we can always eliminate $aa$- or $as$- tiles.  The result then is an annulus $\mathcal{A}$ whose foliation consists of $s$-bands alternating with bigon discs $\Delta_i$ such that two $abs$-tiles serve as the transition on either end of each bigon disc.  Then by a related Euler characteristic calculation (Lemma 6.3.1 of \cite{[BM2]}) for our annulus $\mathcal{A}$, we know that $$V(1,1) + 2V(0,2) + V(0,3) = 2E(s) + V(2,1) + 2V(3,0) + \sum_{v=4}^{\infty} \sum_{\alpha = 0}^v (v+\alpha-4)V(\alpha,v-\alpha)$$ where $E(s)$ is twice the number of $s$-bands; furthermore, if both sides equal zero we know that only vertices of type $(1,2)$ or $(0,4)$ appear.  

If $E(s) > 0$ we may therefore eliminate all vertices in the foliation of $\mathcal{A}$ using braid isotopy, exchange moves and destabilizations, and we obtain an annulus whose foliation consists entirely of $s$-arcs.  If this is the case for all of our $m$ embedded annuli, then we are done, since the resulting $\hat\beta_0^*$ and $\hat\beta_1^*$ will have the same braid index and algebraic length.  

Otherwise, if $E(s)=0$ for some annulus, then again using braid isotopy, exchange moves and destabilizations we either obtain an annulus entirely foliated by $s$-arcs (in which case we are done), or an annulus whose foliation is tiled entirely by $ab$- or $bb$-tiles in which every vertex is either of type $(1,2)$ or $(0,4)$.  If any of the remaining vertices of type $(1,2)$ or $(0,4)$ are adjacent to consecutive singularities of like parity on either end of a one-parameter family of $b$-arcs, then we may perform a standard change of fibration to obtain either a vertex of type $(1,1)$ or $(0,3)$ which we may then remove following braid isotopy and exchange moves.  We may therefore assume we obtain a tiling in which consecutive singularities around any vertex alternate parity (see Figure \ref{fig:checkerboardB}).  The tiling of $\mathcal{A}$ will then be composed of a subannulus of some $r$ number of $ab$-tiles along $\hat\beta_0^*$, and a subannulus of the same $r$ number of $ab$-tiles along $\hat\beta_1^*$, along with $k$ subannuli containing $2r$ number of $bb$-tiles which interpolate between the subannuli of $ab$-tiles; in Figure \ref{fig:checkerboardB}, $r=3$ and $k=1$.  However, in this case the resulting braids $\hat\beta_0^*$ and $\hat\beta_1^*$ will then have the same braid index and algebraic length, and we achieve the desired contradiction.  To justify this last statement, observe that if the annulus $\mathcal{A}$ is oriented so as to agree with $\hat\beta_1^*$, we may stabilize $\hat\beta_1^*$ along singular leaves in $\mathcal{A}$ some $r(k+1)$ number of times to remove all negative vertices in $\mathcal{A}$, and then destabilize along valence-1 vertices the same $r(k+1)$ number of times to remove all positive vertices from $\mathcal{A}$; moreover, the parity of all of these stabilizations and destabilizations are identical, and thus $(\ell(\hat\beta_0^*),n(\hat\beta_0^*)) = (\ell(\hat\beta_1^*),n(\hat\beta_1^*))$. 

\end{proof}

\begin{figure}[htbp]
	\centering
		\includegraphics[width=0.45\textwidth]{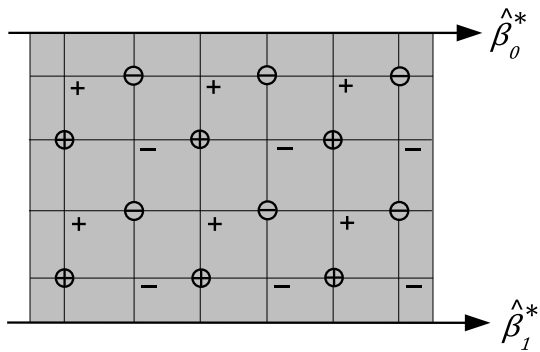}
	\caption{{\small A tiling of $\mathcal{A}$ such that all vertices are either valence-4 or valence-3, and consecutive singularities around any vertex alternate sign.}}
	\label{fig:checkerboardB}
\end{figure}

%\section{Applications of Jones' conjecture}
%\label{sec:apps}

%Given a topological link type $\mathcal{L}$, recall that the classical invariant for transverse links representing $\mathcal{L}$ in the standard tight contact 3-sphere is the {\em self-linking number}, denoted sl, and an invariant for $\mathcal{L}$ is $\overline{\textrm{sl}}$, the maximum self-linking number realized by transverse representatives of $\mathcal{L}$.  Using the fact that any such transverse representative in the standard tight contact 3-sphere is transversely isotopic to a braid representative $\beta \in \mathcal{L}$ and that $\textrm{sl}(\beta)=\ell(\beta)-n(\beta)$ \cite{[B]}, an immediate corollary of Theorem \ref{thm:jones} is therefore:

%\begin{corollary}
%Let $\beta \in \mathcal{L}$ be at minimum braid index; then $$\textrm{sl}(\beta) = \overline{\textrm{sl}}(\mathcal{L})$$  
%\end{corollary}

%A second corollary follows from Theorem \ref{thm:jones} and the following result of Orevkov \cite{[O]}:

%\begin{theorem}[Orevkov]
%A braid $\beta$ is quasi-positive if and only if any positive stabilization of $\beta$ is quasi-positive. 
%\end{theorem} 

%The following corollary was first described to us by J. Etnyre \cite{[E]}:

%\begin{corollary}[Etnyre]
%If $\mathcal{K}$ is a fibered strongly quasi-positive knot, then any minimum braid index representaive of $\mathcal{K}$ is quasi-positive. 
%\end{corollary}

\end{document}